\documentclass[a4paper]{article}
\usepackage{latexsym}
\usepackage{amsmath}
\usepackage{amsthm}
\usepackage{amssymb}
\usepackage{bm}
\usepackage{graphicx}
\usepackage{wrapfig}
\usepackage{fancybox}
\usepackage{url}
\usepackage{setspace}

\newtheorem{theorem}{Theorem}[section]
\newtheorem{proposition}[theorem]{Proposition}

\newtheorem{remark}[theorem]{Remark}
\newtheorem{example}[theorem]{Example}

\numberwithin{equation}{section}

\begin{document}
\title{On some variational algebraic problems}
\author{Giovanni Molica Bisci and Du\v{s}an Repov\v{s}
}
\date{}

\maketitle

\begin{abstract}
In this paper by exploiting critical point theory, the existence of two distinct nontrivial solutions for a nonlinear algebraic system with a parameter is established. Our goal is achieved by requiring an appropriate behavior of the nonlinear term $f$ at zero and at infinity. Some applications to difference equations are also presented.

~

~
\end{abstract}

\section{Introduction}
In this paper we deal with the following problem:\\
\begin{center}
\hfill $Au=\lambda f(u)$, \hfill $(S_{A,\lambda}^{f})$
\end{center}

where $u= (u_{1},\ \ldots\ ,\ u_{n})^{t}\in \mathbb{R}^{n}$ is a column vector in $\mathbb{R}^{n}, A=(a_{ij})_{n\times n}$ is a given positive definite matrix, $f(u) :=(f_{1}(u_{1}),\ \ldots\ ,\ f_{n}(u_{n}))^{t}$ with $f_{k}: \mathbb{R}\rightarrow \mathbb{R}$ a continuous function for every $k\in \mathbb{Z}[1,\ n] :=\{1,\ .\ .\ .\ ,\ n\}$, and $\lambda$ is a positive parameter. 

Discrete problems involving functions with two or more discrete variables are very relevant and have been deeply investigated. Such great interest is undoubtedly due to the advance of modern digital computing devices.

Indeed, since these relations can be simulated in a relatively easy manner by means of such devices and since such simulations often reveal important information about the behavior of complex systems, a large number of recent investigations related to image processing, population models, neural networks, social behaviors, digital control systems are described in terms of such functional relations.

Moreover, a large number of problems can be formulated as special cases of the nonlinear algebraic system $(S_{A,\lambda}^{f})$ . For a survey on these topics we cite the recent paper [25]. A similar approach has also been used in others works (see for instance, the papers [21--23] and [24, 26, 27]).

Here, motivated by the interest on the subject, by using variational methods in finite dimensional setting, we prove the existence of two nontrivial solutions for suitable values of the parameter $\lambda.$

More precisely, in Theorem~\ref{Theorem:3.1} we prove the existence of two nontrivial solutions, for every $\lambda$ sufficiently large, by only requiring sublinear conditions at infinity and an appropriate behaviour of the nonlinear terms at zero.

In Theorem~\ref{Theorem:3.4} we determine an open interval of positive parameters such that problem $(S_{A,\lambda}^{f})$ admits at least two nontrivial solutions which are uniformly bounded in norm with respect to the parameter $\lambda.$

Our main tool, in this case, is a useful abstract result obtained in [3, Theorem 2.1] which ensures the existence of an open interval $\Lambda\subset(0, +\infty)$ such that for each $\lambda\in\Lambda$ the function $J_{\lambda}$ associated to problem $(S_{A,\lambda}^{f})$ admits two critical points which are uniformly bounded in norm with respect to $\lambda$ (see also [5, 6] for related topics).

A direct application of our result to fourth-order difference equations yields the following:

\begin{proposition}\label{Proposition:1.1}
Assume that
$$
\sup_{t\in \mathbb{R}}\sum_{k=1}^{n}\int_{0}^{t}f_{k}(s)ds>0,
$$
in addition to
$$
\lim_{|s|\rightarrow\infty}\frac{f_{k}(s)}{s}=\lim_{s\rightarrow 0}\frac{f_{k}(s)}{s}=0,
$$
for every $k \in \mathbb{Z}[1, n]$. Then there exist a nonempty open interval $\Lambda\subset(0, +\infty)$ and a number 
$\gamma>0$  such that for every $\lambda\in\Lambda$, problem
~\\~

\hfill
$\left\{\begin{array}{l}
\Delta^4 u_{k-2}= \lambda f_{k}(u_{k}), \forall k\in \mathbb{Z}[1, n], \qquad \qquad \qquad \qquad\\
u_{-2}=u_{-1}=u_{0}=0, \hfill (D_{\lambda}^{f})\\
u_{n+1}=u_{n+2}=u_{n+3}=0,
\end{array}\right.$
~\\~\\
has at least two distinct, nontrivial solutions $u_{\lambda}^{1}, u_{\lambda}^{2}\in \mathbb{R}^{n}$,  and
$$
\Vert u_{\lambda}^{i}\Vert_{2}<\gamma,\ i\in\{1,2\}.
$$
\end{proposition}

Further, requiring a suitable growth of the primitive of $f$, we are able to establish suitable intervals of values of the parameter $\lambda$ for which the problem $(S_{A,\lambda}^{f})$ admits at least three weak solutions. More precisely, the main result ensures the existence of two real intervals of parameters $\Lambda_{1}, \Lambda_{2}$ such that, for each $\lambda\in\Lambda_{1}\cup\Lambda_{2},$
the problem $(S_{A,\lambda}^{f})$ admits at least three weak solutions whose norms are uniformly bounded with respect to every $\lambda \in\Lambda_{2}$ (see Theorem~\ref{Theorem:3.5} and Example~\ref{Example:3.7}). Our method is mostly based on a useful critical point theorem given in [4, Theorem 3.1].

In conclusion, we also emphasize that if the functions $f_{k}$ are nonnegative, for every $k\in \mathbb{Z}[1, n]$, our results guarantee two positive solutions (see Remark~\ref{Remark:3.8} for more details). For a complete and exhaustive overview of variational methods we refer the reader to the monographs [1, 16, 19].

The plan of the paper is as follows. In Section~\ref{Section:2} we introduce some basic notations. In Section~\ref{Section:3} we obtain our existence results (see Theorems~\ref{Theorem:3.1} and~\ref{Theorem:3.4}). Finally, in Section~\ref{Section:4}, some concrete examples of applications of the obtained results are presented.

\section{Preliminaries}\label{Section:2}

As the ambient space $X$, we consider the $n$-dimensional Banach space $\mathbb{R}^{n}$ endowed by the norm
$$
\Vert u\Vert_{2}:=\left(\sum_{k=1}^{n}u_{k}^{2}\right)^{1/2}.
$$
More generally, we set
$$
\Vert u\Vert_{r}:=\left(\sum_{k=1}^{n}|u_{k}|^{r}\right)^{1/r}\ (r\geq 1)
$$
for every $u\in X.$

Let $\mathfrak{X}_{n}$ denote the class of all symmetric and positive definite matrices of order $n$. Further, we denote by $\lambda_{1}$, . . . , $\lambda_{n}$ the eigenvalues of $A$ ordered as follows: $0<\lambda_{1}\leq\cdots\leq\lambda_{n}.$

It is well known that if {\it A} $\in\mathfrak{X}_{n}$, for every $u\in X$, then one has
\begin{equation}\label{Equation:2.1}
\lambda_{1}\Vert u\Vert_{2}^{2}\leq u^{t}Au\leq\lambda_{n}\Vert u\Vert_{2}^{2},
\end{equation}

and
\begin{equation}\label{Equation:2.2}
\displaystyle \Vert u\Vert_{\infty}\leq\frac{1}{\sqrt{\lambda_{1}}}(u^{t} Au )^{1/2},
\end{equation}
where $\Vert u\Vert_{\infty} :=\displaystyle \max_{k\in[1,n]}|u_{k}|.$

From now on we will assume that {\it A} $\in\mathfrak{X}_{n}$. Set

\begin{equation}\label{Equation:2.3}
\Phi(u):=\frac{u^{t}Au}{2}, \quad \Psi(u):=\sum_{k=1}^{n}F_{k}(u_{k}) \ \text{ and }\ J_{\lambda}(u):=\Phi(u)- \lambda \Psi(u),
\end{equation}

for every $u\in X$, where

$$F_{k}(t) :=\displaystyle \int_{0}^{t}f_{k}(s)ds, \ \text{ for every } (k, t) \in \mathbb{Z}[1, n]\times \mathbb{R}.$$

Standard arguments show that $J_{\lambda}\in C^{1}(X, \mathbb{R})$ as well as that the critical points of $J_{\lambda}$ are exactly the solutions of problem $(S_{A,\lambda}^{f})$.

Indeed, a column vector $\overline{u}= (\overline{u}_{1}, \ldots , \overline{u}_{n})^{t}\in X$ is a critical point of the functional $J_{\lambda}$ if the gradient of $J_{\lambda}$ at $\overline{u}$ is zero, i.e.,
$$
\frac{\partial J_{\lambda}(u)}{\partial u_{1}}|_{u=\overline{u}}=0,\ \frac{\partial J_{\lambda}(u)}{\partial u_{2}}|_{u=\overline{u}}=0, \ldots , \frac{\partial J_{\lambda}(u)}{\partial u_{n}}|_{u=\overline{u}}=0.
$$
Moreover, for every $k\in\mathbb{Z}[1, n]$, one has that
$$
\frac{\partial u^{t}Au}{\partial u_{k}}=2(Au)_{k},
$$
where $(Au)_{k} :=\displaystyle \sum_{j=1}^{n}a_{kj}u_{j}$. Thus

$$ \frac{\partial J_{\lambda}(u)}{\partial u_{k}}=(Au)_{k}-\lambda f_{k}(u_{k}),\ \text{ for all } k\in \mathbb{Z}[1, n],$$

which yields our assertion.

\section{Main results}\label{Section:3}

Our first result is a multiplicity theorem obtained as a consequence of Tonelli's approach together with a careful analysis of the meaningful Mountain Pass geometry of the functional $J_{\lambda}$. More precisely, we consider the case when the continuous functions $f_{k}: \mathbb{R}\rightarrow \mathbb{R}$ fulfil the following hypotheses:
~\\~\\~
\noindent($\text{h}_1$) 
\textit{For every} $k\in \mathbb{Z}[1, n],$
$$
\lim_{|s|\rightarrow\infty}\frac{f_{k}(s)}{s}=0.
$$
 
\noindent($\text{h}_2$)
\textit{There exists} $\nu_{0}>1$ \textit{such that}
$$
\lim_{|s|\rightarrow 0}\frac{f_{k}(s)}{|s|^{\nu_{0}}}=0,
$$
\textit{for every} $k\in \mathbb{Z}[1, n].$
~\\~\\~
Note that a typical example when ($\text{h}_1$) holds is the following:
~\\~\\~
\noindent($\text{h}_1^{\star}$)
{\it There exist} $q\in(0,1)$ {\it and} $c>0$ {\it such that} $|f_{k}(s)|\leq c|s|^{q}$, {\it for every} $s\in \mathbb{R}.$
~\\~\\~
With the above notations, we can prove the following multiplicity result.

\begin{theorem}\label{Theorem:3.1}
Assume that conditions ($\text{h}_1$) and ($\text{h}_2$) hold in addition to
$$
\sup_{t\in \mathbb{R}}\sum_{k=1}^{n}F_{k}(t)>0.
$$
Then:
\begin{enumerate}
	\item[i.] There exists a positive parameter $\lambda^{\star}$ given by
$$\lambda^{\star}:=\left(\frac{\mathrm{T}\mathrm{r}(A)+2\sum_{i<j}a_{ij}}{2}\right)\left(\max_{t\neq 0}\frac{\sum_{k=1}^{n}F_{k}(t)}{t^{2}}\right)^{-1},$$
such that, for every $\lambda > \lambda^{\star}$, problem 
$(S_{A,\lambda}^{f})$ has at least two distinct, nontrivial solutions $u_{\lambda}^{1}, u_{\lambda}^{2}\in\mathbb{R}^{n}$, where $u_{\lambda}^{1}$ is the global minimum of the energy functional $J_{\lambda}$ associated to $(S_{A,\lambda}^{f})$.

	\item[ii.] If $(\mathrm{h}_{1}^{\star})$  holds, then
$$\Vert u_{\lambda}^{1}\Vert_{2}=o(\lambda^{1/(1-r)}), \ \text{ for every } r\in(q, 1),$$
but
$$\Vert u_{\lambda}^{1}\Vert_{2}\neq O(\lambda^{1/(1-\mu)}), \ \text{ for every } \mu>1,$$
as $\lambda\rightarrow \infty.$
\end{enumerate}
\end{theorem}

\begin{proof}
Due to conditions $(\mathrm{h}_{1})$ and $(\mathrm{h}_{1})$ the term $F_{k}(t)/t^{2}$ tends to zero as $|t|\rightarrow \infty$ and $t\rightarrow 0$, respectively. Moreover, since
$$
\sup_{t\in \mathbb{R}}\sum_{k=1}^{n}F_{k}(t)>0,
$$
there exists $t_{0}\in \mathbb{R}$ such that
$$
\sum_{k=1}^{n}F_{k}(t_{0})>0.
$$
Thus, the value $\lambda^{\star}$ is well-defined. Hence, there exists a number $t_{\star}\in\mathbb{R}\backslash \{0\}$ such that
$$
\frac{\sum_{k=1}^{n}F_{k}(t_{\star})}{t_{\star}^{2}}=\max_{t\neq 0}\frac{\sum_{k=1}^{n}F_{k}(t)}{t^{2}}.
$$
So
$$
\lambda^{\star}:=\left(\frac{\mathrm{T}\mathrm{r}(A)+2\sum_{i<j}a_{ij}}{2}\right)\frac{t_{\star}^{2}}{\sum_{k=1}^{n}F_{k}(t_{\star})}.
$$

At this point fix $\lambda>\lambda^{\star}$ and let us first consider the vector $u^{\star}\in X$ of components $u_{k}^{\star}=t_{\star}$, for every $k\in\mathbb{Z}[1, n].$

One has
\begin{eqnarray*}
J_{\lambda}(u^{\star})= \Phi(u^{\star})- \lambda \Psi(u^{\star})=\left(\frac{\mathrm{T}\mathrm{r}(A)+2\sum_{i<j}a_{ij}}{2}\right)t_{\star}^{2}- \lambda\sum_{k=1}^{n}F_{k}(t_{\star})\\
=(\lambda^{\star}-\lambda)\sum_{k=1}^{n}F_{k}(t_{\star})<0.\qquad \qquad \qquad \qquad \qquad \qquad \qquad \quad \
\end{eqnarray*}

Thus $\displaystyle \inf_{u\in X} J_{\lambda}(u)\leq J_{\lambda}(u^{\star})<0$. Due to ($\text{h}_{1}$), for an arbitrarily $ \epsilon<\frac{\lambda_{1}}{\lambda}$ there exists $c(\epsilon)>0$ such that
$$
F_{k}(t)\leq|F_{k}(t)|\leq\frac{\epsilon}{2}t^{2}+c(\epsilon)|t|,
$$
for every $t\in \mathbb{R}$ and $k\in \mathbb{Z}[1, n]$. Consequently, from the left-hand side of~(\ref{Equation:2.1}), we have
$$
J_{\lambda}(u)\geq(\frac{\lambda_{1}-\lambda\epsilon}{2})\Vert u\Vert_{2}^{2}-\lambda c_{1}c(\epsilon)\Vert u\Vert_{2},
$$
where $c_{1}$ is a positive constant such that $\Vert u\Vert_{1}\leq c_{1}\Vert u\Vert_{2}$, for every $u\in X.$

It follows from this that $J_{\lambda}$ is bounded from below and coercive. Hence, since our ambient space is finite dimensional, the functional $J_{\lambda}$ satisfies the classical compactness (PS)-condition.

Since $J_{\lambda}$ verifies the (PS)-condition and it is bounded from below, by [16, Theorem 1.7], one can fix $u_{\lambda}^{1}\in X$ such that $J(u_{\lambda}^{1})= \inf_{u\in X}J_{\lambda}(u)$. Therefore, $u_{\lambda}^{1}\in X$ is the first solution of $(S_{A,\lambda}^{f})$ and $u_{\lambda}^{1}\neq 0$, since $J_{\lambda}(0_{X})=0$.

Now, we prove that for every $\lambda>\lambda^{\star}$ the functional $J_{\lambda}$ has the standard Mountain Pass geometry. Indeed, by ($\text{h}_{1}$) and ($\text{h}_{2}$), one can fix two constants $\nu >1$ and $C>0$ such that
$$
|F_{k}(t)|\leq C|t|^{\nu +1},
$$
for every $t\in\mathbb{R}$ and $k\in\mathbb{Z}[1, n]$. Moreover, bearing in mind condition~(\ref{Equation:2.1}), one has
$$
J_{\lambda}(u)= \Phi(u)- \lambda \Psi(u)
$$
$$
\qquad \qquad \qquad \geq\frac{\lambda_{1}}{2}\Vert u\Vert_{2}^{2}-\lambda C\Vert u\Vert_{\nu+1}^{\nu+1}
$$
\begin{equation}\label{Equation:3.1}
\qquad \qquad \qquad \qquad  \geq\frac{\lambda_{1}}{2}\Vert u\Vert_{2}^{2}-\lambda c_{\nu+1}^{\nu+1}C\Vert u\Vert_{2}^{\nu+1},
\end{equation}

where $c_{\nu+1}$ is a positive constant such that

$$\Vert u\Vert_{\nu+1}\leq c_{\nu+1}\Vert u\Vert_{2}, \ \text{ for all } u\in X.$$

Let us take $\rho_{\lambda}>0$ to be so small that
$$
\rho_{\lambda}<\min\left\lbrace\left(\frac{\lambda_{1}}{2\lambda c_{v+1}^{v+1}C}\right)^{\frac{1}{v-1}},\ \sqrt{n}|t_{\star}|\right\}.
$$
By~(\ref{Equation:3.1}), for every $u\in X$ complying with $\Vert u\Vert_{2}=\rho_{\lambda}$, we have
\begin{eqnarray*}
J_{\lambda}(u)\geq(\frac{\lambda_{1}}{2}-\lambda c_{v+1}^{v+1}C\Vert u\Vert_{2}^{v-1})\Vert u\Vert_{2}^{2}\\
=(\frac{\lambda_{1}}{2}-\lambda c_{v+1}^{v+1}C\rho_{\lambda}^{v-1})\rho_{\lambda}^{2}\qquad\\
=:\eta(\rho_{\lambda})>0. \qquad \qquad \qquad \; \; \, 
\end{eqnarray*}
By construction, one has $\Vert u\Vert_{2}=\sqrt{n}|t_{\star}|>\rho_{\lambda}$, and $J(u^{\star})<0=J_{\lambda}(0_{X})$ .

Hence, we can apply the Mountain Pass Theorem (see [16, Theorem 1.13]). Thus, there exists $u_{\lambda}^{2}\in X$ such that $J'(u_{\lambda}^{2})=0$ and $J_{\lambda}(u_{\lambda}^{2})\geq\eta(\rho_{\lambda})>0$. Further, $u_{\lambda}^{2}\neq 0_{X}$ and the vectors $u_{\lambda}^{1}$ and $u_{\lambda}^{2}$ are distinct. The proof of point (i) is complete.

Now, we assume that $(\mathrm{h}_{1}^{\star})$ holds. Since $J_{\lambda}(u_{\lambda}^{1})<0$, it follows that
$$
\frac{\lambda_{1}}{2}\Vert u_{\lambda}^{1}\Vert_{2}^{2}-\frac{\lambda c}{(q+1)}c_{q+1}^{q+1}\Vert u\Vert_{2}^{q+1}\leq J_{\lambda}(u_{\lambda}^{1})<0.
$$
In particular, $\Vert u_{\lambda}^{1}\Vert_{2}=O(\lambda^{1/(1-q)})$ as $\lambda\rightarrow \infty$. Therefore, for any $r\in(q, 1)$, one has $\Vert u_{\lambda}^{1}\Vert_{2}=o(\lambda^{1/(1-r)})$ as $\lambda\rightarrow \infty.$

Let us assume that $\Vert u_{\lambda}^{1}\Vert_{2}=O(\lambda^{1/(1-\mu)})$ for some $\mu>1$ as $\lambda\rightarrow \infty$. Then $\Vert u_{\lambda}^{1}\Vert_{2}\rightarrow 0$ as $\lambda\rightarrow \infty$. On the other hand,
$$
J(u_{\lambda}^{1})\leq(\lambda^{\star}-\lambda)\sum_{k=1}^{n}F_{k}(t_{\star})\ ,
$$
hence $J_{\lambda}(u_{\lambda}^{1})\rightarrow-\infty$. Now, since
$$
\frac{\lambda_{1}}{2}\Vert u_{\lambda}^{1}\Vert_{2}^{2}-\lambda c_{\mu+1}^{\mu+1}C\Vert u_{\lambda}^{1}\Vert_{2}^{\mu+1}\leq J_{\lambda}(u_{\lambda}^{1})\ ,
$$
one has
$$
\left(\frac{\lambda_{1}}{2}-\lambda c_{\mu+1}^{\mu+1}C\Vert u_{\lambda}^{1}\Vert_{2}^{\mu-1}\right)\Vert u_{\lambda}^{1}\Vert_{2}^{2}\rightarrow-\infty,
$$
as $\lambda\rightarrow \infty$. This fact contradicts the initial assumption.

The proof is thus complete.
\end{proof} 

\begin{remark}\label{Remark:3.2} 
\normalfont
We observe that Theorem~\ref{Theorem:3.1} can be checked by a careful analysis of a three critical points theorem contained in [7, Theorem 3.6].
\end{remark}

Now, instead of ($\text{h}_2$) we will assume a weaker condition, namely:
~\\~\\~
\noindent($\text{h}_2'$)
$\lim_{s\rightarrow 0}\frac{f_{k}(s)}{s}=0, \ \text{ for every } k\in\mathbb{Z}[1, n].$
~\\~\\~
The next theorem below shows that assumption $(\text{h}_{2}')$ is still strong enough to prove a similar multiplicity result as Theorem~\ref{Theorem:3.1}. In this setting we obtain that the solutions are uniformly bounded in norm with respect to the parameter $\lambda$ but, unfortunately, we lose the precise location of the eigenvalues. The main tool for our goal is a theoretical result given in [3, Theorem 2.1] (see, for completeness, [16, Theorem 1.13]).

We prove the following preliminary fact.

\begin{proposition}\label{Proposition:3.3}
Assume that condition ($\text{h}_1$) holds in addition to ($\text{h}_{2}'$). Then
$$
\lim_{\varrho\rightarrow 0^+}\frac{\sup_{u\in\Phi^{-1}(]-\infty, \varrho[)} \Psi(u)}{\varrho}=0.
$$
\end{proposition}

\begin{proof}
Due to $(\mathrm{h}_{2}')$ , for an arbitrary small $\epsilon>0$ there exists $\delta_{\epsilon}>0$ such that
$$
|f_{k}(s)|<\epsilon|s|,
$$
for every $|s|<\delta_{\epsilon}$ and $k\in \mathbb{Z}[1, n]$. On the other hand, on account of ($\text{h}_{1}$), one can fix $\nu>1$ and
$$
|f_{k}(s)|<\epsilon|s|^{\nu},
$$
for every $|s|\geq\delta_{\epsilon}$ and $k\in \mathbb{Z}[1,n]$. Combining these two facts, we obtain
$$
F_{k}(t)\leq\epsilon\frac{t^{2}}{2}+\frac{c(\epsilon)}{(\nu+1)}|t|^{\nu+1},
$$
for every $t \in\mathbb{R}$ and $k\in\mathbb{Z}[1, n]$.

Now, fix $\varrho>0$. For every $u \in \Phi^{-1} (]-\infty, \varrho [)$, due to the above estimates, we have

$$\displaystyle \Psi(u)\leq\frac{\epsilon}{2}\Vert u\Vert_{2}^{2}+\frac{c(\epsilon)}{(\nu+1)}c_{\nu+1}^{\nu+1}\Vert u\Vert_{2}^{\nu+1}<\frac{\epsilon\varrho}{\lambda_{1}}+c(\epsilon)\frac{c_{\nu+1}^{\nu+1}}{(\nu+1)}(\frac{2\varrho}{\lambda_{1}})^{\frac{\nu+1}{2}},$$

taking into account that
$$
\{u\in X:u^{t}Au\ <2\varrho\}\subset\left\{u\in X:\Vert u\Vert_{2}<\left(\frac{2\varrho}{\lambda_{1}}\right)^{1/2}\right\}.
$$
Thus, there exists $\varrho(\epsilon)>0$ such that, for every $0<\varrho<\varrho(\epsilon)$, we have
$$
0\leq\frac{\sup_{u\in\Phi^{-1}(]-\infty, \varrho[)} \Psi(u)}{\varrho}<\frac{ \epsilon}{\lambda_{1}}+c(\epsilon)\frac{\varrho^{\frac{\nu-1}{2}}}{(\nu+1)}\left(\frac{2}{\lambda_{1}}\right)^{\frac{\nu+1}{2}}< \epsilon,
$$
which completes the proof.
\end{proof}

Our multiplicity result reads as follows.

\begin{theorem}\label{Theorem:3.4}
Assume that conditions ($\text{h}_1$) and ($\text{h}_{2}'$) hold. Then there exist a nonempty open interval $\Lambda\subset(0, +\infty)$ and a number $\gamma>0$ such that for every $\lambda \in \Lambda$, problem $(S_{A,\lambda}^{f})$ has at least two distinct, nontrivial solutions $u_{\lambda}^{1}, u_{\lambda}^{2}\in X$, and $\Vert u_{\lambda}^{i}\Vert_{2}<\gamma, i\in\{1, 2 \}$.
\end{theorem}

\begin{proof}
Let $X :=\mathbb{R}^{n}$, and consider the functionals $\Phi$ and $\Psi$ defined in (\ref{Equation:2.3}). Note that $J_{\lambda} := \Phi - \lambda \Psi$. We already know that for every positive parameter $\lambda$ the functional $J_{\lambda}$ is coercive and consequently satisfies the Palais-Smale condition, because $X$ is finite dimensional.

Due to the fact that the functions $f_{k}$ are sublinear at infinity and superlinear at zero, the terms $F_{k}(t)/t^{2}\rightarrow 0$ as $|t|\rightarrow \infty$ and $t\rightarrow 0$, respectively.

\noindent
Since $\displaystyle \sup_{t\in \mathbb{R}}\sum_{k=1}^{n}F_{k}(t)>0$, there exists $t_{0}\in \mathbb{R}$ such that $\displaystyle \sum_{k=1}^{n}F_{k}(t_{0})>0,$ and we may fix a number $t_{\star}\in \mathbb{R}\backslash \{0\}$ such that
$$
\frac{\sum_{k=1}^{n}F_{k}(t_{\star})}{t_{\star}^{2}}=\max_{t\neq 0}\frac{\sum_{k=}^{n}{}_{1}F_{k}(t)}{t^{2}}.
$$
Therefore the number
$$
\lambda^{\star}:=\left(\frac{\mathrm{T}\mathrm{r}(A)+2\sum_{i<j}a_{ij}}{2}\right)\frac{t_{\star}^{2}}{\sum_{k=}^{n}{}_{1}F_{k}(t_{\star})}
$$
is well-defined.

Now, let us choose $u^{0}=0_{X}$ and $u^{1}\in X$ such that

$$u_{k}^{1}=t_{\star}, \ \text{ for every } k \in\mathbb{Z}[1, n].$$

Fixing $\epsilon \in (0,1)$ , due to Proposition~\ref{Proposition:3.3}, one can choose $\varrho>0$ such that

$$\frac{\sup_{u\in\Phi^{-1}(]-\infty},\varrho[) \Psi(u)}{\varrho}<\frac{\epsilon}{\lambda^{\star}} \text{ and } \varrho< \left(\frac{\mathrm{T}\mathrm{r}(A)+2\sum_{i<j}a_{ij}}{2}\right)t_{\star}^{2}.$$

Note that

$$\frac{\epsilon}{\lambda^{\star}}<\frac{1}{\lambda^{\star}}=\frac{ \Psi(u^{1})}{\Phi(u^{1})} \text{ and }\Phi(u^{1})=\left(\frac{\mathrm{T}\mathrm{r}(A)+2\sum_{i<j}a_{ij}}{2}\right)t_{\star}^{2}.$$

Therefore, by choosing
$$
\overline{a} := \frac{1+ \epsilon}{\frac{ \Psi(u^{1})}{\Phi(u^{1})}-\frac{\sup_{u\in\Phi^{-1}(]-\infty,\varrho[)} \Psi(u)}{\varrho}},
$$
all the assumptions of [3, Theorem 2.1] can be verified.

Hence, there exist a non-empty open interval $\Lambda\subset[0, \overline{a}]$ and a positive real $\gamma$ such that for every $\lambda \in\Lambda$, the functional $J_{\lambda}$ admits at least three distinct critical points in $X$ having $\Vert\cdot\Vert_{2}$-norm less than $\gamma$. The proof is complete.
\end{proof}

As a direct application of [4, Theorem 3.1] we give the following multiplicity property.

\begin{theorem}\label{Theorem:3.5}
Let $f_{k}: \mathbb{R}\rightarrow \mathbb{R}$ be a continuous function for every $k\in\mathbb{Z}[1, n]$. Assume that there exist positive constants $\gamma$ and $\delta$ such that

\begin{enumerate}

	\item[($\text{g}_1$)] $\delta>\left(\frac{\lambda_{1}}{\mathrm{T}\mathrm{r}(A)+2\Sigma_{i<j} a_{ij}}\right)^{1/2}\gamma$.

	\item[($\text{g}_2$)] The following inequality holds:
$$
\sum_{k=1}^{n}\max_{|\xi|\leq\gamma} F_{k}(\xi) < \eta(\gamma,\delta)\ \left(\sum_{k=1}^{n}F_{k}(\delta)\right),
$$
where
$$
\eta(\gamma,\delta)\ :=\frac{\lambda_{1}\gamma^{2}}{\lambda_{1}\gamma^{2}+(\mathrm{T}\mathrm{r}(A)+2\sum_{i<j}a_{ij})\delta^{2}}.
$$
\end{enumerate}
Further require that

\begin{enumerate}
	\item[($\text{g}_3$)] $\lim\sup_{|\xi|\rightarrow\infty}\frac{F_{k}(\xi)}{\xi^{2}}<\frac{\lambda_{1}}{2} \ \text{ for all } k\in \mathbb{Z}[1, n]$.
\end{enumerate}
Then, for each
$$
\lambda\in\Lambda_{1}:=]\lambda_{1}^{\star},\ \lambda_{2}^{\star}[,
$$
where
$$
\lambda_{1}^{\star}:=\frac{\mathrm{T}\mathrm{r}(A)+2\sum_{i<j}a_{ij}}{2(\sum_{k=1}^{n}F_{k}(\delta)-\sum_{k=1}^{n}\max_{|\xi|\leq)\prime}F_{k}(\xi))},
$$
and
$$
\lambda_{2}^{\star}:=\frac{\lambda_{1}\gamma^{2}}{2(\sum_{k=1}^{n}\max_{|\xi|\leq\gamma}F_{k}(\xi))},
$$
problem $(S_{A,\lambda}^{f})$ has at least three distinct solutions and, moreover for each $h>1$, there exists an open interval
$$
\Lambda_{2}\subset[0,\ \lambda_{3,h}^{\star}],
$$
where
$$
\lambda_{3,h}^{\star}:=\frac{\lambda_{1}h\gamma^{2}}{2(\lambda_{1}\gamma^{2}(\frac{\Sigma_{k--1}^{n}F_{k}(\delta)}{\mathrm{T}\mathrm{r}(A)+2\Sigma_{i<j^{0}ij}})-\sum_{k=1}^{n}\max_{|\xi|\leq\gamma}F_{k}(\xi))},
$$
and a positive real number $\sigma>0$ such that, for each $\lambda\in\Lambda_{2}$, problem $(S_{A,\lambda}^{f})$ has at least three solutions whose norms are less than $\sigma$.
\end{theorem}

\begin{proof}
We use the notations adopted above. Our aim is to apply [4, Theorem 3.1]. First of all let us verify that $J_{\lambda}$ is a coercive functional for every positive parameter $\lambda$. By ($\text{g}_3$) there are constants $\epsilon \in ] 0, \lambda_{1}/2[$ and $\sigma>0$ such that
\begin{equation}
\displaystyle \frac{1}{\xi^{2}}\int_{0}^{\xi}f_{k}(s)ds<\frac{\lambda_{1}}{2}-\epsilon
\end{equation}
for every $|\xi|\geq\sigma$ and $k\in \mathbb{Z}[1,n]$. Let us put
\begin{equation}
M_{1} := \max_{(k,\xi)\in \mathbb{Z}[1,n]\times[-\sigma,\sigma]} \int_{0}^{\xi}f_{k}(s)ds.
\end{equation}
At this point note that, for every $\xi\in\mathbb{R}$ and $k\in\mathbb{Z}[1, n]$, one has
$$
\int_{0}^{\xi}f_{k}(s)ds\leq M_{1}+M_{2}\xi^{2},
$$
where
$$
M_{2}:=\frac{\lambda_{1}}{2}-\epsilon.
$$
Moreover, the following inequality holds:
$$J_{\lambda}(u)\displaystyle \geq\frac{u^{t}Au}{2}-\sum_{k=1}^{n}[M_{1}+M_{2}u_{k}^{2}], \ \text{ for all } u\in X.$$

Hence,
$$J_{\lambda}(u) \geq\frac{u^{t}Au}{2}-M_{2}\Vert u\Vert_{2}^{2}-nM_{1}, \ \text{ for all } u\in X,$$

and by relation~(\ref{Equation:2.1}), one has
\begin{equation}
J_{\lambda}(u)\geq\epsilon\Vert u\Vert_{2}^{2}-nM_{1}, \ \text{ for all } u\in X,
\end{equation}
which clearly shows that
\begin{equation}\label{Equation:3.5}
\lim_{\Vert u\Vert_{2}\rightarrow\infty}  J_{\lambda}(u)=+\infty.
\end{equation}

Hence $J_{\lambda}$ is coercive for every positive parameter $\lambda>0$.

Next, consider the vector $u^{\star}\in X$ of components $u_{k}^{\star} = \delta$, for every $k\in\mathbb{Z}[1,n]$. Thus
\begin{equation}\label{Equation:3.6}
\Phi(u^{\star})=\left(\frac{\mathrm{T}\mathrm{r}(A)+2\sum_{i<j}a_{ij}}{2}\right)\delta^{2}.
\end{equation}
Put
$$
r:=\frac{\lambda_{1}}{2}\gamma^{2}.
$$

It follows now from ($\text{g}_1$) that $\Phi(u^{\star})>r$. Further, we explicitly observe that, in view of~(\ref{Equation:2.2}), one has
\begin{equation}\label{Equation:3.7}
\Phi^{-1}(]-\infty, r[)\subset\{u\in X:\Vert u\Vert_{\infty}\leq \gamma\}.
\end{equation}
Moreover, taking~(\ref{Equation:3.7}) into account, a direct computation ensures that
\begin{equation}\label{Equation:3.8}
\sup_{u \in \Phi^{-1} (]-\infty,r[)} \Psi(u)\leq\sum^{n}_{k=1}\max_{|\xi|\leq y}F_{k}(\xi).
\end{equation}

At this point, by definition of $u^{\star}$, we can clearly write
\begin{equation}\label{Equation:3.9}
\displaystyle \Psi(u^{\star})=\sum_{k=1}^{n}F_{k}(u^{\star})=\sum_{k=1}^{n}F_{k}(\delta).
\end{equation}
Moreover, by using hypothesis ($g_2$) from~(\ref{Equation:3.8}) and~(\ref{Equation:3.9}), we have
$$
\sup_{u \in \Phi^{-1} (]-\infty,r[)} \Psi(u) < \frac{r}{r+\Phi(u^{\star})} \Psi(u^{\star}),
$$

taking into account that
$$
\frac{r}{r+\Phi(u^{\star})}= \eta( \gamma, \delta).
$$
Thus, we can apply [4, Theorem 3.1], bearing in mind that
$$
\frac{\Phi(u^{\star})}{\Psi(u^{\star})- \sup_{u \in \Phi^{-1} (]-\infty,r[)} \Psi(u)} \leq \lambda_{1},
$$
and
$$
\frac{r}{\sup_{u \in \Phi^{-1} (]-\infty,r[)} \Psi(u)} \geq \lambda_{2}
$$
as well as
$$
\frac{hr}{r \frac{\Psi(u^{\star})}{\Phi(u^{\star})} - \sup_{u \in \Phi^{-1} (]-\infty,r[)} \Psi(u)} \leq \lambda_{3,h}^{\star}.
$$
The proof is complete.
\end{proof}

\begin{remark}
\normalfont
As observed in [4, Remark 2.1], the real intervals $\Lambda_{1}$ and $\Lambda_{2}$ in Theorem~\ref{Theorem:3.5} are such that either
$$
\Lambda_{1}\cap\Lambda_{2}=\emptyset,
$$
or
$$
\Lambda_{1}\cap\Lambda_{2}\neq\emptyset.
$$
In the first case, we actually obtain two distinct open intervals of positive real parameters for which problem $(S_{A,\lambda}^{f})$ admits two nontrivial solutions; otherwise, we obtain only one interval of positive real parameters, precisely $\Lambda_{1}\cup\Lambda_{2}$, for which problem $(S_{A,\lambda}^{f})$ admits three solutions and in addition, the subinterval $\Lambda_{2}$ for which the solutions are uniformly bounded.
\end{remark}

The following is a simple application of Theorem~\ref{Theorem:3.5}.

\begin{example}\label{Example:3.7}
\normalfont
Let $g_{k}: \mathbb{R}\rightarrow \mathbb{R}$ be as follows:

$$g_{k}(s):=\left\{\begin{array}{ll}
0 & \mathrm{i}\mathrm{f}\ s<2,\\
k\sqrt{s-2} & \mathrm{i}\mathrm{f}\ s\geq 2,
\end{array}\right.$$

whose potentials are given by

$$G_{k}(t):=\displaystyle \int_{0}^{t}g_{k}(s)ds=\left\{\begin{array}{ll}
0 & \mathrm{i}\mathrm{f}\ t<2,\\
\frac{2k}{3}(t-2)^{3/2} & \mathrm{i}\mathrm{f}\ t\geq 2,
\end{array}\right.$$

for every $k\in \mathbb{Z}[1, n]$. Consider the algebraic nonlinear system
\begin{center}
\hfill $Au=\lambda g(u)$, \hfill $(S_{A,\lambda}^{g})$
\end{center}
where $A \in\mathfrak{X}_{n}$ and $g(u) :=(g_{1}(u_{1}), \ldots , g_{n}(u_{n}))^{t}$.

We observe that there exist two positive constants $\gamma=2$ and
$$
\delta > 2 \max \left\{ 1, \left(\frac{\lambda_{1}}{\mathrm{T}\mathrm{r}(A)+2\sum_{i<j}a_{ij}}\right)^{1/2}\right\},
$$

such that all the conditions of Theorem~\ref{Theorem:3.5} hold. Then for each
$$
\lambda\in\Lambda_{1}':=]\lambda_{1}^{\star},\ +\infty[,
$$
where
$$
\lambda_{1}^{\star}:=\frac{\mathrm{T}\mathrm{r}(A)+2\sum_{i<j}a_{ij}}{2(\sum_{k=1}^{n}G_{k}(\delta))},
$$
problem $(S_{A,\lambda}^{g})$ has at least three distinct solutions (two nontrivial) and moreover, for each $h>1$, there exist an open interval
$$
\Lambda_{2}'\subset[0, \lambda_{3,h}^{\star}],
$$
where
$$
\lambda_{3,h}^{\star}:=h\frac{\mathrm{T}\mathrm{r}(A)+2\sum_{i<j}a_{ij}}{2(\sum_{k=1}^{n}G_{k}(\delta))}=h\lambda_{1}^{\star},
$$
and a positive real number $\sigma>0$ such that for each $\lambda\in\Lambda_{2}'$, problem $(S_{A,\lambda}^{g})$ has at least three solutions whose norms are less than $\sigma.$
\end{example}

\begin{remark}\label{Remark:3.8}
\normalfont
A vector $\overline{u}:= (\overline{u}_{1}, \ldots, \overline{u}_{n})^{t}2\mathbb{R}^{n}$ is said to be {\it positive} ({\it nonnegative}) if $\overline{u}_{k}>0$ ($\overline{u}_{k}\geq 0$) for every $k\in\mathbb{Z}[1, n]$. Now, let $A \in\mathfrak{X}_{n}$ and consider the following conditions:
\begin{enumerate}
	\item[($\text{A}_1$)] If $i\neq j$, then $a_{ij}\leq 0$.

	\item[($\text{A}_2$)] For every $i\in\mathbb{Z}[2, n]$, there exists $j_{\mathrm{i}}<i$ such that $a_{ij_{i}}<0$.
\end{enumerate}
Assume that ($\text{A}_1$) holds. Then, if $\overline{u} :=(\overline{u}_{1}, \ldots ,\overline{u}_{n})^{t}\in X$ is a solution of

\begin{center}
\hfill $\sum_{j=1}^{n}a_{ij}u_{j}\geq 0 \ \text{ for all } i\in \mathbb{Z}[1, n],$ \hfill $(S_{A}^{\star})$
\end{center}

then $\overline{u}_{i} \geq 0$, for every $i\in\mathbb{Z}[1, n]$ (see [11, 28] and [9, Proposition 2.1]). If, in addition to ($\text{A}_1$), condition ($\text{A}_2$) holds, then any solution of $(S_{A}^{\star})$ is trivial or otherwise is positive (see [9, Proposition 2.2]). Hence, if $f_{k}$ are nonnegative, for every $k\in \mathbb{Z}[1, n]$, our results guarantee the existence of two nonnegative solutions if $A$ satisfies hypothesis ($\text{A}_1$). Finally, if ($\text{A}_2$) holds together with ($\text{A}_1$), then the obtained solutions are positive.
\end{remark}

\section{Applications}\label{Section:4}
In this section we present some direct applications to discrete equations.

\subsection{Tridiagonal matrices}

Let $n>1$ and $(a, b)\in\mathbb{R}^{-}\times \mathbb{R}^{+}$ be such that
$$
\cos\left(\frac{\pi}{n+1}\right)<-\frac{b}{2a}.
$$
Set

$$\text{Trid}_n (a, b, a) =
\begin{pmatrix}
b & a & 0 & \ldots & 0\\
a & b & a & \ldots & 0\\
  &   & \ddots  &  &  \\
0 & \ldots & a & b & a\\
0 & \ldots & 0 & a & b\\
\end{pmatrix}
_{n\times n}
$$

Note that $\text{Trid}_n (a, b, a)$ is a symmetric and positive definite matrix whose first eigenvalue is given by
$$
\lambda_{1}=b+2a\cos\left(\frac{\pi}{n+1}\right),
$$
see, for instance, [20, Example 9, page 179]. In this setting an important case is given by the following matrix:

$$\text{Trid}_n (-1, 2, -1) =
\begin{pmatrix}
2 & -1 & 0 & \ldots & 0\\
-1 & 2 & -1 & \ldots & 0\\
  &   & \ddots  &  &  \\
0 & \ldots & -1 & 2 & -1\\
0 & \ldots & 0 & -1 & 2\\
\end{pmatrix}
_{n\times n}
\in \mathfrak{X}_{n},
$$

which is associated to the second-order discrete boundary value problem
~\\~
\begin{center}
\hfill
$\left\{\begin{array}{l}
-\Delta^2 u_{k-1} = \lambda f_{k}(u_{k}),\qquad \forall k\in \mathbb{Z}[1,n], \qquad \qquad \qquad \qquad\\
\qquad \quad \, u_{0}=u_{n+1}=0, \hfill (S_{\lambda}^{j}) \\
\end{array}\right.$
\end{center}
~\\~
where $\Delta^2 u_{k-1} := \Delta (\Delta u_{k-1})$, and, as usual, $\Delta u_{k-1} := u_{k}-u_{k-1}$ denotes the forward difference operator. We point out that the matrix $\text{Trid}_n (-1, 2, -1)$ was considered in order to study the existence of nontrivial solutions of nonlinear second-order difference equations [8, 14, 15, 17]. For completeness, we just men- tion here that there is a vast literature on nonlinear difference equations based on fixed point and upper and lower solution methods (see, for instance, the papers [2, 12]).

\begin{example}\label{Example:4.1}
\normalfont
By Theorem~\ref{Theorem:3.4}, there are a non-empty open interval $\Lambda\subset(0, +\infty)$ and a number $\gamma>0$ such that for every $\lambda \in\Lambda$, the following problem,

\begin{center}
\hfill $\text{Trid}_n(a, b, a)u=\lambda g(u)$, \hfill $(T_{\lambda}^{g})$
\end{center}

where $g(u):=(g_{1}(u_{1}), \ldots , g_{n}(u_{n}))^{t}$, in which

$$g_{\mathrm{i}}(u_{i}):=\left\{\begin{array}{ll}
-iu_{i}^{2} & \mathrm{i}\mathrm{f}\ u_{\mathrm{i}}\leq 0,\\
\frac{iu_i}{\log u_{i}} & \mathrm{i}\mathrm{f}\  0 < t \leq e^{i},\\
\frac{i}{e} & \mathrm{i}\mathrm{f}\  u_i > e^{i},\\
\end{array}\right.$$

has at least two distinct nontrivial solutions $u_{\lambda}^{1}, u_{\lambda}^{2}\in \mathbb{R}^{n}$, and
$$
\Vert u_{\lambda}^{i}\Vert_{2}<\gamma, \quad i\in\{1, 2\}.
$$
Note that $g$ in Example~\ref{Example:4.1} satisfies ($\text{h}_{2}'$) but not ($\text{h}_2$) for any constant $v_{0}>1$. Therefore, one can apply Theorem~\ref{Theorem:3.4} but not Theorem~\ref{Theorem:3.1}.
\end{example}

\subsection{Fourth-order difference equations}
As it is well known, boundary value problems involving fourth-order difference equations such as
~\\~
\begin{center}
\hfill
$\left\{\begin{array}{l}
\Delta^4 u_{k-2}= \lambda f_{k}(u_{k}), \quad \forall k\in \mathbb{Z}[1, n], \qquad \qquad \qquad \qquad\\
\quad \; \; \, u_{-2}=u_{-1}=u_{0}=0, \hfill (D_{\lambda}^{f})\\
\quad \; u_{n+1}=u_{n+2}=u_{n+3}=0,
\end{array}\right.$
\end{center}
~\\~
can also be expressed as the problem $(S_{A,\lambda}^{f})$ , where $A$ is the real symmetric and positive definite matrix of the form
$$A :=
\begin{pmatrix}
6 & -4 & 1 & 0 & \ldots & 0 & 0 & 0 & 0\\
-4 & 6 & -4 & 1 & \ldots & 0 & 0 & 0 & 0\\
1 & -4 & 6 & -4 & \ldots & 0 & 0 & 0 & 0\\
0 & 1 & -4 & 6 & \ldots & 0 & 0 & 0 & 0\\
  &   &    &   & \ddots &   &   &   & \\
0 & 0 & 0 & 0 & \ldots & 6 & -4 & 1 & 0\\
0 & 0 & 0 & 0 & \ldots & -4 & 6 & -4 & 1\\
0 & 0 & 0 & 0 & \ldots & 1 & -4 & 6 & -4\\
0 & 0 & 0 & 0 & \ldots & 0 & 1 & -4 & 6\\
\end{pmatrix}
_{n\times n}
\in \mathfrak{X}_{n}.
$$
Hence, Proposition~\ref{Proposition:1.1} is a direct consequence of Theorem~\ref{Theorem:3.4}.

\subsection{Partial difference equations}
A {\it lattice point} $z :=(i, j)$ in the plane is a point with integer coordinates. Two lattice points are said to be {\it neighbors} if their Euclidean distance is one. An edge is a set $\{z, z^{\star}\}$ consisting of two neighboring points, whereas a directed edge is an ordered pair $(z, z^{\star})$ of neighboring points. A {\it path} between two lattice points $z$ and $z^{\star}$ is a sequence $z=z_{0}, \ldots , z_{s} = z^{\star}$ of lattice points such that $z_{\mathrm{i}}$ and $z_{i+1}$ are neighbors for $0\leq i\leq s-1$. A set $S$ of lattice points is said to be {\it connected} if there is a path contained in $S$ between any two points of $S$. A finite and connected set of lattice points is called a {\it net}. An {\it exterior boundary} point of a net $S$ is a point outside $S$ but has a neighbor in $S$. The set of all exterior boundary points is denoted by $\partial S$. The set of all edges of $S$ is denoted by $\Gamma(S)$ and the set of all directed edges of a net $S$ by $E(S)$ . The pair $(S,\Gamma(S))$ is a planar graph and the pair $(S, E(S))$ is a planar directed graph. With the above notations we consider the problem, namely $(E_{\lambda}^{f})$, given by

$$\left\{\begin{array}{rl}
Du(z)+\lambda f(z,\ u(z))=0, & z \in S,\\
u(z)=0, & z \in \partial S,\\
\end{array}\right.$$

where

$$Du(z):=[u(i+1, j)-2u(i, j)+u(i-1, j)]+[u(i, j+1)-2u(i, j)+u(i, j-1)]$$

is the well-known discrete Laplacian acting on a function $u: S\cup \partial S\rightarrow \mathbb{R}$. Then problem $(E_{\lambda}^{f})$ can be written as a nonlinear algebraic system (see, for more details, the monograph of Cheng [10]). We also cite the paper [13] in which the existence of infinitely many solutions for problem $(E_{\lambda}^{f})$ has been investigated.

\begin{example}\label{Example:4.2} 
\normalfont
For each
$$
\lambda>\frac{1}{0.3787311542}-\sim 2.6,
$$
the following problem,
\begin{eqnarray*}
[u(i+1, j)-2u(i, j)+u(i-1, j)]+[u(i, j+1)-2u(i, j)+u(i, j-1)]\\
+\lambda a(u(i, j))=0, \forall(i, j)\in \mathbb{Z}[1, 2]\times \mathbb{Z}[1, 2], \qquad \qquad \qquad \qquad
\end{eqnarray*}
with boundary conditions
$$
u(i,0) = u(i,3) = 0, \quad \forall i\in\mathbb{Z}[1, 2],
$$
$$
u(0, j)=u(3, j)=0, \quad \forall j\in\mathbb{Z}[1, 2],
$$
where $a(s) :=\log(1+s^{2})$ for every $s>0$ and zero otherwise, admits two non- trivial (positive) solutions.

Indeed, let $h: \mathbb{Z}[1, 2]\times \mathbb{Z}[1, 2]\rightarrow \mathbb{Z}[1, 4]$ be the bijection defined by

$$h(i,\ j):=i+2(j-1), \ \text{ for every } (i, j)\in \mathbb{Z}[1, 2]\times \mathbb{Z}[1, 2].$$

Next, define
$$
w_{k}:=u(h^{-1}(k)),
$$
and
$$
g_{k}(w_{k})=g_{k}(u(h^{-1}(k))):=a(w_{k}),
$$
for every $k\in\mathbb{Z}[1, 4]$. The above problem can then be written as
$$
Bw=\lambda g(w),
$$
where
$$B :=
\begin{pmatrix}
4 & -1 & \vdots & -1 & 0\\
-1 & 4 & \vdots & 0 & -1\\
\ldots & \ldots & \ldots & \ldots & \ldots  \\
-1 & 0 & \vdots & 4 & -1\\
0 & -1 & \vdots & -1 & 4\\
\end{pmatrix}
,
$$
$w:= (w_{1},\ldots, w_{k})$ and $g(w):=(g_{1}(w_{1}), \ldots , g_{4}(w_{4}))^{t}$. Our assertion now immediately follows from Theorem~\ref{Theorem:3.1} and Remark~\ref{Remark:3.8}.
\end{example}

Some recent results about the discontinuous case were obtained in [18].
~\\~

\textbf{Acknowledgments.} This paper was written when the first author was visiting  the University of Ljubljana in  2012. He expresses his gratitude for the warm hospitality.
This research was supported  by the SRA grants P1-0292-0101 and J1-4144-0101.

\section*{Bibliography}
\setlength{\parskip}{1em}
\noindent
[1] R. P. Agarwal, {\it Difference Equations and Inequalities}: {\it Theory Methods and Applications}, Marcel Dekker, New York, 2000.

\noindent
[2] C. Bereanu and J. Mawhin, Existence and multiplicity results for nonlinear second order difference equations with Dirichlet boundary conditions, {\it Math. Bohem}. 131 (2006), 145-160.

\noindent
[3] G. Bonanno, Some remarks on a three critical points theorem, {\it Nonlinear Anal}. 54 (2003), 651-665.

\noindent
[4] G. Bonanno, A critical points theorem and nonlinear differential problems, {\it J. Global Optim}. 28 (2004), 249-258.

\noindent
[5] G. Bonanno and P. Candito, Infinitely many solutions for a class of discrete non-linear boundary value problems, {\it Appl. Anal}. 88 (2009), 605-616.

\noindent
[6] G. Bonanno and P. Candito, Nonlinear difference equations investigated via critical point methods, {\it Nonlinear Anal}. 70 (2009), 3180-3186.

\noindent
[7] G. Bonanno and S. A. Marano, On the structure of the critical set of non-differentialble functions with a weak compactness condiction, {\it Appl. Anal}. 89 (2010), 1-10.

\noindent
[8] P. Candito and G. Molica Bisci, Existence of two solutions for a nonlinear second-order discrete boundary value problem, {\it Adv. Nonlinear Stud}. 11 (2011), 443-453.

\noindent
[9] P. Candito and G. Molica Bisci, Existence of solutions for a nonlinear algebraic system with a parameter, {\it Appl. Math. Comput}. 218 (2012), 11700-11707.

\noindent
[10] S. S. Cheng, {\it Partial Difference Equations}, Taylor\& Francis, London, 2003.

\noindent
[11] M. Fiore, A proposito di alcune disequazioni lineari, {\it Rend. Semin. Mat. Univ Padova} 17 (1948), 18.

\noindent
[12] J. Henderson and H. B. Thompson, Existence of multiple solutions for second order discrete boundary value problems, {\it Comput. Math. Appl}. 43 (2002), 1239-1248.

\noindent
[13] M. Imbesi and G. Molica Bisci, Discrete elliptic Dirichlet problems and nonlinear algebraic systems, preprint.

\noindent
[14] A. Krist\'{a}ly, M. $\mathrm{M}\mathrm{i}\mathrm{h}\dot{\mathrm{a}}$ilescu and V $\mathrm{R}\dot{\mathrm{a}}$dulescu, Discrete boundary value problems involving oscillatory nonlinearities: Small and large solutions, {\it J. Difference} $Equ.$ {\it Appl}. 17 (2011), 1431-1440.

\noindent
[15] A. Krist\'{a}ly, M. $\mathrm{M}\mathrm{i}\mathrm{h}\dot{\mathrm{a}}$ilescu, V. $\mathrm{R}\dot{\mathrm{a}}$dulescu and S. Tersian, Spectral estimates for a nonhomogeneous difference problem, {\it Commun. Contemp. Math}. 12 (2010), no. 6, 1015-1029.

\noindent
[16] A. Krist\'{a}ly, V $\mathrm{R}\dot{\mathrm{a}}$dulescu and C. Varga, {\it Variational Principles in Mathematical Physics, Geometry and Economics}: {\it Qualitative Analysis of Nonlinear Equations and Unilateral Problems}, Encyclopedia of Mathematics and Its Applications 136, Cambridge University Press, Cambridge, 2010.

\noindent
[17] M. $\mathrm{M}\mathrm{i}\mathrm{h}\dot{\mathrm{a}}$ilescu, V. $\mathrm{R}\dot{\mathrm{a}}$dulescu and S. Tersian, Eigenvalue problems for anisotropic discrete boundary value problems, {\it J. Difference} $Equ$. {\it Appl}. 15 (2009), 557-567.

\noindent
[18] G. Molica Bisci and D. Repov\v{s}, Nonlinear algebraic systems with discontinuous terms, {\it J. Math. Anal}., 398:2 (2013), 846-856.

\noindent
[19] D. Motreanu and V. $\mathrm{R}\dot{\mathrm{a}}$dulescu, {\it Variational and Non-Variational Methods in NonlinearAnalysis and Boundary Value Problems}, Nonconvex Optimization and Its Applications 67, Kluwer Academic, Dordrecht, 2003.

\noindent
[20] J. T. Scheick, {\it Linear Algebra with Applications}, International Series in Pure and Applied Mathematics, McGrawHill, New York, 1997.

\noindent
[21] G. Wang and S. S. Cheng, Elementary variational approach to zero-free solutions of a non linear eigenvalue problem, {\it NonlinearAnal}. 69 (2008), 3030-3041.

\noindent
[22] Y. Yang and J. Zhang, Existence results for a nonlinear system with a parameter, {\it J. Math. Anal. Appl}. 340 (2008), no. 1, 658-668.

\noindent
[23] Y. Yang and J. Zhang, Existence and multiple solutions for a nonlinear system with a parameter, {\it NonlinearAnal}. 70 (2009), no. 7, 2542-2548.

\noindent
[24] G. Zhang, Existence of non-zero solutions for a nonlinear system with a parameter, {\it NonlinearAnal}. 66 (2007), no. 6, 1410-1416.

\noindent
[25] G. Zhang and L. Bai, Existence of solutions for a nonlinear algebraic system, {\it Discrete} $Dyn$. {\it Nat. Soc}. 2009 (2009), Article ID 785068.

\noindent
[26] G. Zhang and S. S. Cheng, Existence of solutions for a nonlinear algebraic system with a parameter, {\it J. Math. Anal. Appl}. 314 (2006), 311-319.

\noindent
[27] G. Zhang and W. Feng, On the number of positive solutions of a nonlinear algebraic system, {\it Linear Algebra Appl}. 422 (2007), 404-421.

\noindent
[28] G. Zwirner, Criteri d'unicit\`{a} per un problema di valori al contorno per equazioni $\mathrm{e}$ sistemi di equazioni differenziali ordinarie d'ordine qualunque, {\it Rend. Semin. Mat. Univ Padova} 13 (1942), 9-25.
\\~\\~\\
Received October 18, 2012; accepted November 5, 2012.
\\~\\
\noindent
\textbf{Author information}
~\\[0.6em]
\noindent
Giovanni Molica Bisci, Dipartimento MECMAT, University of Reggio Calabria, Via Graziella, Feo di Vito, 89124 Reggio Calabria, Italy.\\
\noindent
E-mail: \url{gmolica@unirc.it}
~\\[0.6em]
\noindent
Du\v{s}an Repov\v{s}, Faculty of Education, and Faculty of Mathematics and Physics, University of Ljubljana, Kardeljeva plo\v s\v cad 16, 1000 Ljubljana, Slovenia.\\
\noindent
E-mail: \url{dusan.repovs@guest.arnes.si}

\end{document}